\documentclass[11pt,reqno]{amsart}

\usepackage{amssymb,amsfonts,amsaddr,mathrsfs,bm,enumerate}
\usepackage{mathtools,bbm,mathptmx,color}
\usepackage{verbatim}

\usepackage{txfonts}

\makeatletter
\newtheorem{theorem}{Theorem}[section]
\newtheorem{lemma}[theorem]{Lemma}
\newtheorem{corollary}[theorem]{Corollary}

\theoremstyle{definition} \newtheorem{remark}[theorem]{Remark}

\allowdisplaybreaks
\numberwithin{equation}{section}

\newbox\ovlbox
\def\ovl#1{\setbox\ovlbox\hbox{$#1$}\rlap{\kern.5\wd\ovlbox\kern-1.5pt
  $\overline{\hbox to4pt{\hss$\phantom{#1}$\hss}}$\hss}#1}
\def\unl#1{\setbox\ovlbox\hbox{$#1$}\rlap{\kern.5\wd\ovlbox\kern-2.5pt
  $\underline{\hbox to4pt{\hss$\phantom{#1}$\hss}}$\hss}#1}

\def\Ric{{\operatorname{Ric}}} \def\Hess{\operatorname{Hess}}
\def\Sect{{\operatorname{Sect}}} 
 \def\II{{\operatorname{II}}}

\def\tr{\operatorname{tr}}
\def\div{{\operatorname{div}}}

\newcommand{\eps}{\varepsilon}
\newcommand\1{\mathbbm{1}}

 \newcommand\E{\mathbb{E}}
 
 \newcommand\R{\mathbb{R}}

\def\Ric{{\operatorname{Ric}}} 
 \def\Hess{\operatorname{Hess}}
\def\Eig{\operatorname{Eig}} \def\II{{\operatorname{II}}}

\def\mathpal#1{\mathop{\mathchoice{\text{\rm #1}}%
    {\text{\rm #1}}{\text{\rm #1}}%
    {\text{\rm #1}}}\nolimits} \def\id{{\mathpal{id}}}

\def\bd{{\bf d}}
 \def\r{\right}
\def\l{\left} \def\e{\operatorname{e}} 
 
 \newcommand{\ptr}{/\!/}
\def\II{{\operatorname{II}}}
 \def\id{{\mathpal{id}}}

  \def\mathpal#1{\mathop{\mathchoice{\text{\rm #1}}%
      {\text{\rm #1}}{\text{\rm #1}}%
      {\text{\rm #1}}}\nolimits}

  \makeatother

\begin{document}

 \title[Hessian estimate for eigenfunctions]
  {Hessian estimates for Dirichlet and Neumann eigenfunctions of Laplacian}

  \author{Li-Juan Cheng\textsuperscript{1}*, \quad Anton Thalmaier\textsuperscript{2},\quad Feng-Yu Wang\textsuperscript{3} }

  \address{\textsuperscript{1}School of  Mathematics, Hangzhou Normal University\\
    Hangzhou 311123, The People's Republic of China}
 \address{\textsuperscript{2} Department of Mathematics, University of
      Luxembourg, Maison du Nombre,\\
      L-4364 Esch-sur-Alzette, Luxembourg}
  \address{\textsuperscript{3} Center for Applied Mathematics, Tianjin
      University,\\
       Tianjin 300072, People's Republic of China}

  \email{lijuan.cheng@hznu.edu.cn}
  \email{anton.thalmaier@uni.lu}
  \email{wangfy@tju.edu.cn}
\thanks{*Corresponding author}


\begin{abstract}
  By methods of stochastic analysis on Riemannian manifolds, we develop
  two approaches to determine an explicit constant $c(D)$ for an
  $n$-dimensional compact manifold $D$ with boundary such that
  \begin{align*}
    \frac{\lambda}{n}\,\|\phi\|_{\infty} \leq \|\Hess\phi\|_\infty\leq c(D)\lambda \,\|\phi\|_{\infty}
  \end{align*}
  holds for any Dirichlet eigenfunction $\phi$ of $-\Delta$ with
  eigenvalue $\lambda$.  Our results provide the sharp Hessian estimate
$\|\Hess \phi\|_{\infty}\lesssim 
  \lambda^{\frac{n+3}{4}}$.
  Corresponding Hessian estimates for Neumann
  eigenfunctions are derived in the second part of the paper.
\end{abstract}

\date{\today}

\maketitle

\section{Introduction\label{s1}}
Let $D$ be an $n$-dimensional compact Riemannian manifold with smooth 
boundary $\partial D$. We write
$(\phi, \lambda) \in \Eig(\Delta)$ if $\phi$ is a
Dirichlet eigenfunction of $-\Delta$ on $D$ with eigenvalue
$\lambda>0$, i.e.~$-\Delta\phi=\lambda\phi$.
We always assume eigenfunctions $\phi$ to be normalized in
$L^2(D)$ such that $\|\phi\|_{L^2}=1$. According to \cite{SX}, there exist two positive constants
$c_{1}(D)$ and $c_{2}(D)$ such that
\begin{align}\label{esti-eigenfunctions}
  c_{1}(D) \sqrt{\lambda}\,\|\phi\|_{\infty} \leqslant\|\nabla \phi\|_{\infty}
  \leqslant c_{2}(D) \sqrt{\lambda}\,\|\phi\|_{\infty}, \quad(\phi,\lambda)\in\Eig(\Delta),
\end{align}
where we write
$\|\nabla \phi\|_{\infty}:=\|\,|\nabla \phi|\,\|_{\infty}$ for
simplicity.  An analogous statement for Neumann eigenfunctions has
been derived by Hu, Shi and Xui~\cite{HSX}. Subsequently, by methods
of stochastic analysis on Riemannian manifolds, Arnaudon, Thalmaier
and Wang \cite{ATW} determined explicit constants $c_1(D)$ and $c_2(D)$ in
\eqref{esti-eigenfunctions} for Dirichlet and Neumann
eigenfunctions. From this, together with the uniform estimate of $\phi$ (see
\cite{Hormander,Grieser,Levitan}),
\begin{align*}
  \|\phi\|_{\infty}\leq c_D\lambda^{\frac{n-1}{4}}
\end{align*}
for some positive constant $c_D$,
the optimal uniform bound of the gradient writes as
\begin{align*}
  \|\nabla\phi\|_{\infty}\lesssim \lambda ^{\frac{n+1}{4}}.
\end{align*}
Results of this type have been used to study gradient estimates for unit
spectral projection operators and to give a new proof of
H\"{o}rmander's multiplier theorem, see \cite{XuPhDThesis,
  Xu2007,Xu2009}.

Concerning higher order estimates of eigenfunctions, not much is known.
Very recently, Steinerberger \cite{Ster} studied Laplacian
eigenfunctions of $-\Delta $ with Dirichlet boundary conditions on bounded
domains $\Omega \subset \R^n$ with smooth boundary and proved a sharp
Hessian estimate for the eigenfunctions which reads as
\begin{align*}
  \|\Hess \phi\|_{\infty}\lesssim 
  \lambda^{\frac{n+3}{4}}
\end{align*}
where $$ \|\Hess \phi\|_{\infty}:=\sup\l\{|\Hess \phi(v,v)|(x): x\in \R^n,\ v\in \R^n, \  |v|=1\r\}.$$
To the best of our knowledge, higher order estimates of eigenfunctions for Euclidean domains
first appeared in \cite{Frank-Seiringer} (see Lemma C.1 in the Appendix there
which is easily adapted to cover our situation).

It is a natural question under which geometric assumptions such estimates
extend to compact manifolds (with boundary).
Following the lines of \cite{ATW}, one may ask the question 
how for the Hessian to derive explicit numerical constants $C_1(D)$ and $C_2(D)$ such that
\begin{align}\label{aim-ineq}
  C_{1}(D) {\lambda}\,\|\phi\|_{\infty} \leqslant\|\Hess \phi\|_{\infty} \leqslant C_{2}(D) {\lambda}\,\|\phi\|_{\infty}, \quad(\phi,\lambda) \in \Eig(\Delta).
\end{align}
Note that for eigenfunctions of the Laplacian, one trivially has
\begin{align*}
  |\Hess \phi|\geq \frac1n\,|\Delta \phi|=\frac\lambda{n}\,|\phi|,
\end{align*}
and thus there is always the obvious lower bound
\begin{align*}
\frac{\|\Hess \phi\|_{\infty}}{\|\phi\|_{\infty}}\geq \frac\lambda{n}.
\end{align*}
For this reason, we concentrate in the sequel on upper bounds for
${\|\Hess\phi\|_{\infty}}/{\|\phi\|_{\infty}}$.

In \cite{ATW} a derivative formula for Dirichlet eigenfunctions has been
given from where an upper bound for the gradient of the 
eigenfunction could be derived directly.
Let us briefly describe this method. Assume that $X_t$ is a Brownian motion
on $D\setminus\partial D$ with generator $\frac12\Delta$, and
write $X_t(x)$ to indicate the starting point $X_0=x$.
Then $X{\bf.}(x)$ is defined up to the first hitting time
$\tau_D=\inf\{t>0\colon X_t(x)\in\partial D\}$  of the boundary.
For $x\in\partial D$ we use the convention that $X{\bf.}(x)$ is defined with lifetime
$\tau_D\equiv0$; in this case the subsequent statements usually hold automatically.

Suppose that $Q_t\colon T_xD\rightarrow T_{X_t(x)}D$ is defined by
\begin{align*}
  \text{D}Q_t=-\frac12\Ric^{\sharp}(Q_t)\, d t, \quad  Q_0=\id,
\end{align*}
where $\text{D}:=\ptr_t\,d \,\ptr_t^{-1}$ with
$\ptr_t:=\ptr_{0,t}\colon T_{x}D \rightarrow T_{X_{t}(x)}D$ parallel
transport along $X(x)$ and $\Ric^{\sharp}(v)(w)=\Ric(v,w)$ for $v,w \in TD$. Suppose that
$(\phi,\lambda) \in \Eig(\Delta)$.  Then, for $v\in T_xM$ and any
$k\in C^1_b([0,\infty);\R)$, i.e.,
$k$ bounded with bounded derivative, the process
\begin{align*}
  k(t)\e^{\lambda t/2}\,\langle \nabla \phi (X_t),Q_t(v)\rangle-\e^{\lambda t/2}\phi(X_t)\int_0^t
  \langle \dot{k}(s)Q_s v,\ptr_s dB_s \rangle,\quad t\leq \tau_D
\end{align*}
is a martingale.  From this, by taking expectation, a formula
involving $\nabla\phi$ can be obtained which allows to derive an upper
bound for $|\nabla \phi|$ on $D$ by estimating $|\nabla \phi|$ on the
boundary $\partial D$ and carefully choosing the function $k$.  Along
this circle of ideas, our aim is to establish a similar strategy
for the Hessian of an eigenfunction $\phi$.

In view of the fact that $P_t\phi=\e^{-\lambda t/2}\phi$ where $P_t$
is the semigroup generated by $\frac12\Delta$, we focus first on
martingales which are appropriate for attaining uniform Hessian
estimates of eigenfunctions.  Let us start with some background on
Bismut type formulas for second-order derivatives of heat
semigroups. A second-order differential formula for the heat semigroup
$P_t$ was first obtained by Elworthy and Li \cite{EL94, Li} for a
non-compact manifold, however with restrictions on the curvature of
the manifold.  An intrinsic formula for $\Hess P_tf$ has been given by
Stroock \cite{Stroock} for a compact Riemannian manifold, and a
localized version of such a formula was obtained in \cite{APT, ChCT} adopting 
martingale arguments.  For the Hessian of the Feynman-Kac
semigroup of an operator $\Delta+V$  with a potential function $V$ on manifolds, we refer the reader to \cite{Li2018,Li2016,Th19}. 

For a complete Riemannian manifold $M$ without boundary,
an appropriate version of a Bismut-type Hessian formula gives the following estimate
(see \cite[Corollary 4.3]{ChCT} and Lemma \ref{th3}, or Corollary \ref{th-esti-HS}  with $\sigma_1=\sigma_2=0$):
\begin{align*}
\|\Hess P_tf\|_{\infty}\leq  \l(K_1 \sqrt{t}+\frac{K_2 t}{2}\r)\e^{K_0t}\,\|f\|_{\infty}+\frac{2}{t}\e^{K_0t}\, \|f\|_{\infty}
\end{align*}
where 
\begin{align}\label{K012}
&K_0:=\sup\l\{-{\rm Ric}(v,v)\colon y\in M,\ v\in T_yM,\ |v|=1\r\};\notag\\
& K_1:=\sup\l\{|R|(y)\colon y\in M \r\};\\
&K_2:=\sup\l\{|(\bd^*R+\nabla \Ric)^{\sharp}(v,w)|(y)\colon y\in M,\ v,w\in T_yM,\ |v|=|w|=1\r\} \notag
\end{align}
and 
\begin{align*}
|R|(y):=\sup \l\{\sqrt{\sum_{i,j=1}^nR(e_i, v, w,e_j)^2(y)}: |v|\leq 1, |w|\leq 1\r\}
\end{align*}
for an orthonormal base $\{e_i\}_{i=1}^n$ of $T_yM$.\smallskip 

Thus if $f=\phi$ and $(\phi,\lambda)\in \Eig(\Delta)$, then 
\begin{align*}
\|\Hess \phi\|_{\infty}\leq \l(K_1 \sqrt{t}+\frac{K_2 t}{2}\r)\e^{(K_0+\lambda/2)t}\|\phi\|_{\infty}+\frac{2\e^{(K_0+\lambda/2)t}}{t}\, \|\phi\|_{\infty}
\end{align*}
for any $t>0$. Letting $t=\frac{1}{K_0+\lambda/2}$ then yields the estimate
\begin{align*}
\frac{\|\Hess \phi\|_{\infty}}{\|\phi\|_{\infty}}\leq \l(K_1\sqrt{\frac{2}{2K_0+\lambda}}+ \frac{K_2}{2K_0+\lambda}\r)\e+(\lambda+2K_0)\e.
\end{align*}

To carry over such results to (compact) manifolds $D$ with boundary, the
influence of the boundary has to be studied.  In this paper, we
shall adopt a martingale approach to the Hessian of
Dirichlet eigenfunctions. This approach is based on the construction of a suitable
martingale which builds a relation between $\Hess\phi$ and ${\bf d}\phi$ and then
to estimate $C_2(D)$ in \eqref{aim-ineq} by searching for explicit
constants $C_1$, $C_2$ and $C_3$ such that
\begin{align}\label{Hess-ineq}
  \|\Hess\phi\|_{\infty}\leq C_1\|\Hess \phi\|_{\partial D,\infty}+C_2 \|\nabla\,\phi\|_{\partial D,\infty}+C_3\|\nabla\,\phi\|_{\infty}
\end{align}
where
$\|\Hess \phi\|_{\partial D, \infty}:=\sup_{x\in \partial D} |\Hess
\phi|(x)$ and
$\|\nabla\, \phi\|_{\partial D, \infty}:=\sup_{x\in \partial D}
|\nabla\, \phi|(x)$.  The final estimate for $|\Hess \phi|$ is then
received by combining the last inequality with estimate
\eqref{esti-eigenfunctions} in \cite{ATW}.

Let us start with the general principle behind the construction of the relevant
martingale.  Let $k\in C_b^1([0,\infty);\R)$  and define an
operator-valued process $W^k_t\colon T_xD\otimes T_xD\rightarrow T_{X_t(x)}D$ as
solution to the following covariant It\^{o} equation
\begin{align*}
 {\rm D} W_t^k(v, w)=R(\ptr_td  B_t, Q_t(k(t)v))Q_t(w)-\frac12(\bd^*R+\nabla \Ric)^{\sharp}(Q_t(k(t)v), Q_t(w))\,d  t-\frac12\Ric^{\sharp}(W_t^k(v, w))\,d  t,
\end{align*}
with initial condition $W_0^k(v,w)=0$. Here the operator $\bd ^*R$ is defined by
$\bd^*R(v_1,v_2):=-\tr\nabla{\!\bf.}\,R(\cdot, v_1)v_2$ and thus
satisfies
\begin{align*}
  \langle \bd^*R(v_1,v_2),v_3\rangle=\langle(\nabla_{v_3}\Ric^{\sharp})(v_1),v_2 \rangle-\langle (\nabla_{v_2}\Ric^{\sharp})(v_3),v_1 \rangle
\end{align*}
for all $v_1,v_2,v_3\in T_xD$ and $x\in D$.  Then the process
\begin{align}\label{martingale1}
M_t&:=\e^{\lambda t/2} \Hess \phi \big(Q_t(k(t)v), Q_t(v)\big)+\e^{\lambda t/2} \bd \phi
    (W_t^k(v, v)) \notag\\
  &\quad -\e^{\lambda t/2} \bd \phi(Q_t(v))\int_0^t\langle Q_s(\dot{k}(s)v),
    \ptr_sd  B_s \rangle
\end{align}
is a martingale on $[0,\tau_D]$ in the sense that $(M_{t\wedge\tau_D})_{t\geq0}$ is a globally defined martingale where $\tau_D=\inf\{t>0:X_t(x)\in\partial D\}$ denotes the first hitting time of
$X{\bf.}(x)$ of the boundary $\partial D$. The martingale property of \eqref{martingale1} then allows to establish
an inequality of the type \eqref{Hess-ineq} by equating the expectations at time 0 and at time
$t\wedge\tau_D$. This approach then requires to estimate the boundary values
of $|\bd \phi|$ and $|\Hess \phi|$, in order to obtain the wanted upper
bound for $\|\Hess \phi\|_{\infty}$. To this end, we establish the required 
estimates in Lemmas 2.4-2.5 by using the information on the second fundamental form $\II$ and the second derivative of  $N$, where for
$X, Y\in T_x\partial D$ and $x\in \partial D$, the second fundamental form is defined by
\begin{align*}
  \II(X,Y)=-\langle \nabla_X N, \ Y \rangle. 
\end{align*}
Finally, let
\begin{align}\label{fun-h}
\ell(t):=\ell_{k,\sigma}(t):=\left\{
      \begin{array}{ll}
        \cos{\sqrt{k}t}-\frac{\sigma}{\sqrt{k}}\sin {\sqrt{k}t}, & k> 0, \\
         1-\sigma t, & k= 0, \\
        \cosh{\sqrt{-k}t}-\frac{\sigma}{\sqrt{-k}}\sinh{\sqrt{-k}t}, & k<0,
      \end{array}
    \right.
\end{align}
We state now the first main result of this paper.

 \begin{theorem}\label{main-theorem1}
Let $D$ be a compact Riemannian manifold with smooth
boundary $\partial D$. Let $K_0, K_1, K_2, \sigma$ be non negative constants such that  $\Ric\geq -K_0$, $|R|\leq K_1$ and $|\bd^* R+\nabla \Ric|\leq K_2$ on $D$, and  that $|\II|\leq \sigma$.  Assume that the distance function  $\rho_{\partial}$ is smooth on the tubular neighborhood $\partial_{r_1}D:=\{x\in D: \rho_{\partial }(x)\leq r_1\}$ of $\partial D$. Let $k, \beta, \gamma$ be constants such that  $|\Sect|\leq k$ on $\partial_{r_1}D$, and that
  \begin{align}\label{condition-rho}
    |\nabla (\Delta \rho_{\partial })|\leq \beta,\quad |\Delta^2 \rho_{\partial}|\leq \gamma  \quad 
    \mbox{on}\ \partial_{r_0}D,
  \end{align}
  where $r_0=r_1\wedge \ell^{-1}({1}/{2})$. 
 Then for any non-trivial
$(\phi,\lambda) \in \Eig_N(\Delta)$, 
\begin{align*}
\frac{\|\Hess \,\phi\|_{\infty}}{\|\phi\|_{\infty}}\leq C_{\lambda}(D)\lambda
\end{align*}
where
\begin{align}\label{add-Hess-D1}
C_{\lambda}(D)
&\leq  2(n-1)\e\sigma  \l(\frac{\alpha}{\lambda}+\sqrt{\frac{2}{\pi \lambda}}\r)+{2\alpha}\e^{\frac12 n \sigma r_0+\frac{1}{2}} \max \l\{\sqrt{\frac{1}{\lambda}+\frac{2K_0}{\lambda^2}+\frac{\sigma}{\lambda^2} \left(\frac{n}{r_0}+2\sigma\right)}\,,\, \frac{2\e^{\frac12+\frac12 n \sigma r_0}}{\lambda}\left( \frac{6}{r_0}+2\alpha\right)\r\} \notag\\
&\quad + \frac{\frac{3}{r_0}(\alpha^2+2\beta)+ \frac{6}{r_0^2}\alpha +\gamma}{ \left(\frac{6}{r_0}+2\alpha\right)\lambda}+\frac{\alpha}{\frac{6}{r_0}+2\alpha}\notag\\
&\quad+\sqrt{\e}\l(\frac{\frac{9}{r_0} \alpha + \frac{6}{r_0^2}  + 3\beta+\lambda  +K_1}{ \frac{6}{r_0}+2\alpha}  +\frac{K_2}{4\sqrt{\e}\left( \frac{6}{r_0}+2\alpha\right)^2}\r) \l(\frac{\alpha}{\lambda} +\l(\frac{1}{4}\sqrt{\frac{\pi}{2}}+\sqrt{\frac{2}{\pi}}\r)\sqrt{\frac{1}{\lambda}+\frac{K_0}{\lambda^2}}\r) \notag\\
  &\quad + 4\e^{\frac12 n \sigma r_0+\frac12}  \max \l\{\sqrt{1+\frac{2K_0}{\lambda}+\frac{\sigma}{\lambda} \left(\frac{n}{r_0}+2\sigma\right)}\,,\, \frac{2\e^{\frac12+\frac12 n \sigma r_0}}{\sqrt{\lambda}}\left( \frac{6}{r_0}+2\alpha\right)\r\}\\
  &\qquad\qquad\qquad\qquad\times \l(\frac{\alpha}{\sqrt{\lambda}} +\l(\frac{1}{4}\sqrt{\frac{\pi}{2}}+\sqrt{\frac{2}{\pi}}\r)\sqrt{1+\frac{K_0}{\lambda}}\r),\notag
\end{align}
for $\alpha=2(n-1)\max\{\sigma, k\}$.
\end{theorem}

\begin{remark}\label{rem0}
1)  Adopting the estimate above, obviously $C_{\lambda}(D)$ is
   decreasing in $\lambda$, and hence
  $C_{\lambda}(D)\leq C_{\lambda_1}(D)$
  where $\lambda_1$ is the first Dirichlet eigenvalue of $-\Delta$
  which gives
  \begin{align*}
   \frac{\|\Hess \phi\|_{\infty}}{\|\phi\|_{\infty}}\leq C_{\lambda_1}(D)\lambda.
  \end{align*}
 
2) If the manifold has constant sectional curvature and mean curvature on $\partial_{r_0}D$, i.e. $H=\theta,\,  \Sect=k$ on $\partial_{r_0}D$, then for $\rho_{\partial}(x)\leq \ell^{-1}(0)\wedge r_0$,
 $$\Delta \rho_{\partial}=\frac{\ell'_{\theta, (n-1)k}}{\ell_{\theta, (n-1)k}}(\rho_{\partial}).$$
 As a consequence, the upper bound of $|\nabla (\Delta \rho_{\partial })|$ and  $|\Delta^2 \rho_{\partial}|$ can be calculated explicitly, as
 \begin{align*}
 |\nabla (\Delta \rho_{\partial })|(x)\leq 4((n-1)k+\sigma^2),\ \ |\Delta^2 \rho_{\partial}|(x)\leq 8\max\big\{\sigma, \sqrt{(n-1)k}\big\}((n-1)k+\sigma^2),
 \end{align*}
 for $\rho_{\partial} (x)\leq i_0\wedge \ell^{-1}({1}/{2})$.
 
 For the general  case, from the  second variation formula of $\rho_{\partial}$ (see \eqref{variation-formula} below) we see that
 further information about $|\nabla \II|$, $|\nabla^2\II|$, $|R|$, $|\nabla R|$ and $|\nabla^2 R|$ on $\partial_{r_0}D$ is needed to derive an upper bound of 
  $|\nabla (\Delta \rho_{\partial })|$ and  $|\Delta^2 \rho_{\partial}|$.
 \end{remark}

Turning now to Hessian estimates for Neumann eigenfunctions, let us 
denote by $\Eig_N(\Delta)$ the set of non-trivial
$(\phi,\lambda)$ for the Neumann eigenproblem, i.e., $\phi$ is
non-constant, $\Delta \phi=-\lambda \phi$ and $N\phi|_{\partial D}=0$
for the unit inward normal vector field $N$ of $\partial D$. Proceeding along the 
previous ideas, the main difference is that we can no longer consider the process only 
up to the first hitting the boundary $\partial D$. When constructing the suitable
martingales, the boundary behaviour of the process must be included a priori.
We will use the reflecting Brownian motion as base process to
 deal with this question.  Due to recent work on Bismut-type Hessian formula for the Neumann
 semigroup \cite{CT22}, we have  the following formula linking $\Hess P_tf$ and $\bd f$ intrinsically:
  \begin{equation*}
    \Hess {P_tf}(v,v)=\E\l[-\bd f(\tilde{Q}_t(v))\int_0^t\langle \tilde{Q}_s(\dot{k}(s)v),
    \ptr_s d  B_s\rangle+ \bd f( \tilde{W}_t^k(v,v))\r],
  \end{equation*}
where $\tilde{Q}$ and $\tilde{W}^k$ are defined in \eqref{eq-tildeQ} and \eqref{eq-tildeW} in Section 
\ref{Neuman-Section}. 
By observing the fact that $P_t\phi=\e^{-\frac{1}{2}\lambda t}\phi$ and estimating $\tilde{Q}{\boldsymbol{.}}$ and $\tilde{W}{\!\boldsymbol{.}}$ carefully under suitable curvature conditions, we obtain the following
theorem which gives an upper estimate for $\Hess\phi$ of the type \eqref{aim-ineq}
with an explicit constant $C_2(D)$.

\begin{theorem}\label{Neumann-them}
Let $D$ be an $n$-dimensional compact Riemannian manifold with
boundary $\partial D$. Let $K_0, K_1, K_2$  be non-negative constants such that $\Ric\geq- K_0$, $|R|\leq K_1$ and
$|\bd^* R+\nabla \Ric|\leq K_2$ on $D$,and let $\sigma_1,\sigma_2,\sigma$ be non-negative constants such that
$-\sigma_1\leq \II\leq \sigma$ and $|\nabla^2 N-R(N)|\leq \sigma_2$ on the
boundary $\partial D$. Assume the distance function $\rho_{\partial }$
to the boundary $\partial D$ is smooth on $\partial_{r_1}D:=\{x\in D: \rho_{\partial }(x)\leq r_1\}$ and let $k$ be constant such that $\Sect\leq k$ on $\partial_{r_1}D$. Then for any non-trivial
$(\phi,\lambda) \in \Eig_N(\Delta)$, 
\begin{align*}
\frac{\|\Hess \,\phi\|_{\infty}}{\|\phi\|_{\infty}}\leq C_{N,\lambda}(D)\lambda
\end{align*}
where
\begin{align*}
C_{N,\lambda}(D)=&\l(1+\frac{K_1+2K_0+2\sigma_1 \left(\frac{n}{r_0}+2\sigma_1\right)}{\lambda}+\frac{K_2+2\sigma_2 \left(\frac{n}{r_0}+2\sigma_1\right)}{\lambda\sqrt{2\lambda+4K_0+4\sigma_1\left(\frac{n}{r_0}+2\sigma_1\right)}}\r)\e^{\frac{3}{2}\sigma_1 n r_1+1}\\
&+\frac{ \sigma_2 n r_0}{2\lambda}\sqrt{2\lambda+4K_0+4\sigma_1\left(\frac{n}{r_0}+2\sigma_1\right)}\e^{\frac{3}{2}\sigma_1 n r_0+1}
\end{align*}
for $r_0=r_1\wedge \ell^{-1}(0)$.
Denoting by $\lambda_1$ the first Neumann eigenvalue of $-\Delta$, then 
\begin{align*}
\frac{\|\Hess \phi\|_{\infty}}{\|\phi\|_{\infty}}\leq C_{N,\lambda_1}(D) \,\lambda.
\end{align*}
\end{theorem}

The remainder of the paper is organized as follows. In Section \ref{Dirichlet-section} we first
show for Dirichlet eigenfunctions
\begin{align}\label{main-ineq1}
\|\Hess \phi\|_{\infty}/\|\phi\|_{\infty}\leq C_{\lambda}(D)\lambda
\end{align} by verifying that
the process \eqref{martingale1} is a martingale, in combination with boundary estimates
for $|\Hess \phi|$. 
Section \ref{Neuman-Section} deals with Neumann eigenfunctions where we give a proof of Theorem \ref{Neumann-them} by using Bismut type Hessian formulae for the Neumann semigroup and an estimate of the local time.

\section{Hessian estimates of Dirichlet eigenfunctions }\label{Dirichlet-section}
This section is  dedicated to  the approach 
described in the Introduction. In fact, the proof of Theorem \ref{main-theorem1}
is  also divided into two steps by first showing  Theorem \ref{main-th0} with some
auxiliary function $h$, which will be constructed in Section \ref{Section-Construction-of-h}.

\subsection{Preliminary}
We start by defining the fundamental martingale which will serve as basis for our method.

\begin{theorem}\label{th1}
  On a compact Riemannian manifold $D$ with
  boundary $\partial D$, let $X{\bf.}(x)$ be a Brownian motion starting
  from $x\in D$ and denote by $\tau_D=\inf\{t\geq0\colon X_t(x)\in\partial D\}$ its first
  hitting time of $\partial D$. Define $Q_t$ and $W_t^k$ as above where
  $k\in C_b^1([0,\infty);\R)$.  Then,
  for $(\phi,\lambda) \in \Eig_N(\Delta)$ and $v\in T_xD$, the process
  \begin{align}\label{martingale-1}
    &\e^{\lambda t/2} \Hess \phi \big(Q_t(k(t)v), Q_t(v)\big)+\e^{\lambda t/2} \bd \phi
      (W_t^k(v, v))\notag\\ 
     &\quad-\e^{\lambda t/2} \bd \phi(Q_t( v))\int_0^t\langle Q_s(\dot{k}(s)v),
     \ptr_sd  B_s \rangle
  \end{align}
  is a martingale on $[0,\tau_D]$.
\end{theorem}

\begin{proof} Due to the compactness of $D$ it is sufficient to check that
\eqref{martingale-1} is a local martingale on $[0,\tau_D)$.
  Fixing a time $T>0$, for $v\in T_xD$, we let 
  \begin{align*}
    N_t(v,v)= \Hess P_{T-t}\phi(Q_t(v), Q_t(v))+(\bd P_{T-t} \phi)(W_t(v,v)),\quad
              t\leq T\wedge\tau_D,
  \end{align*}
where 
\begin{align*}
  W_{t}(v,v)=Q_t\int_0^tQ_{r}^{-1}R\big(\ptr_{r} d B_r, Q_{r}(v)\big)Q_{r}(v)-\frac12 Q_t\int_0^t Q_{r}^{-1}(\bd^* R+\nabla \Ric)^{\sharp}\big(Q_{r}(v), Q_{r}(v)\big)\, d r.
\end{align*}
  Then $N_t(v,v)$ is a local martingale, see for instance the proof of
  \cite[Lemma 2.7]{Thompson2019} with potential $V\equiv 0$.
  Since $(\phi,\lambda)\in \Eig(\Delta)$, we know that
  $P_{T-t}\phi(X_t)=\e^{-\lambda(T-t)/2}\phi(X_t)$ and thus 
  $$\e^{\lambda t/2}\Hess \phi(Q_t(v), Q_t(v))+\e^{\lambda t/2}(\bd \phi)(W_t(v,v))$$
is also a local martingale.
Furthermore, consider
$$N^k_t(v,v):=\e^{\lambda t/2} \Hess \phi(Q_t(k(t)v), Q_t(v))+(\e^{\lambda t/2} \bd \phi)(W^k_t(v,v)).$$ 
According to the definition of $W^k_t(v,v)$, resp.~$W_t(v,v)$, and in
view of the fact that $N_t(v,v)$ is a local martingale, it is easy to
see that
\begin{align*}
  \e^{\lambda t/2} \Hess \phi(Q_t(k(t)v), Q_t(v))+(\e^{\lambda t/2} \bd\phi)(W^k_t(v,v))-\int_0^t \e^{\lambda s/2}\Hess \phi(Q_s(\dot{k}(s)
              v), Q_s(v))\,ds
\end{align*}
is a local martingale as well. From the formula
\begin{align*}
  \e^{\lambda t/2} \bd \phi(Q_t(v))=\bd \phi(v)+\int_0^t \e^{\lambda s/2}(\Hess \phi)
  (\ptr_sd  B_s, Q_s(v))
\end{align*}
it follows that
\begin{align}\label{local-M3a}
  &\int_0^t \e^{\lambda s/2}(\Hess \phi)(Q_s(\dot{k}(s)v), Q_s(v))
    \,ds-\e^{\lambda t/2} \bd \phi(Q_t(v))\int_0^t\langle Q_s(\dot{k}(s)v),
    \ptr_sd  B_s\rangle
\end{align}
is a local martingale. We conclude that
\begin{align*}
 (\e^{\lambda t/2} \Hess \phi)(Q_t(k(t)v), Q_t(v))+(\e^{\lambda t/2} \bd \phi)(W^k_t(v,v))-
 \e^{\lambda t/2} \bd \phi(Q_t(v))\int_0^t\langle Q_s(\dot{k}(s)v),
   \ptr_sd  B_s\rangle
\end{align*}
is a local martingale.
\end{proof}

 We shall use the following estimate to proceed with the Hessian formula for $\phi$.

 \begin{lemma}\label{th3}
Assume that  $\Ric\geq -K_0$, $|R|\leq K_1$ and $|\bd^* R+\nabla \Ric|\leq K_2$ on $D$ for non-negative constants $K_0, K_1$ and $K_2$.   
Let $k\in C_b^{1}([0,\infty);\R)$. For $t\geq 0$ and $\delta>0$,
it holds
\begin{align}
  &|Q_t|\leq \e^{K_0t/2}\quad\text{and} \label{Q-est1}\\
  &\E\l[\big|W_{t}^k(v,\dot{k}(t)v)\big|\1_{\{t\leq\tau_D\}}\r]
\leq \l(K_1 \l(\int_0^tk(s)^2\,ds\r)^{1/2}+\frac{K_2}{2}\int_0^t |k(s)|\, ds\r)\,\e^{K_0t} \, |\dot{k}(t)|, \label{Q-est2}
\end{align}
where  $K_0,K_1$ and $K_2$ are defined as in \eqref{K012}.
\end{lemma}
\begin{proof}
The first inequality follows from the lower Ricci curvature bound condition and the definition of $Q_t$.
According to the definition of $W_t^k$, it is easy to see that 
\begin{align*}
W_t^k(v,v)=&Q_t\int_0^t Q_s^{-1} R(\ptr_s d B_s, Q_s(k(s)v))Q_s(v)\\
&-\frac{1}{2}Q_t\int_0^tQ_s^{-1}(\bd^*R+\nabla \Ric)^{\sharp}(Q_s(k(s)v), Q_s(v))\, ds.
\end{align*}
Note that for $0\leq s\leq t$, the damped parallel transport $Q_{s,t}=Q_{t}Q_s^{-1}\colon T_{X_s}D \rightarrow T_{X_t}D$ satisfying 
\begin{align*}
DQ_{t,s}=-\frac{1}{2}\Ric^{\sharp}(Q_{t,s})\, dt, \qquad Q_{s,s}=\id,
\end{align*}
Thus the lower bound of Ricci curvature  $-K_0$  yields
$$|Q_{s,t}|\leq \e^{K_0(t-s)/2}.$$
Then we have
\begin{align}\label{L1-Wt}
\E\l(|W_t^k(v,v)| \1_{\{t\leq \tau_D\}}\r) 
&\leq \E\l[\1_{\{t\leq \tau_D\}} \big|Q_t\int_0^t Q_s^{-1} R(\ptr_s d B_s, Q_s(k(s)v))Q_s(v)\big|\r] \notag\\
&\quad + \frac{1}{2} \E\l[\1_{\{t\leq \tau_D\}}\big|Q_t\int_0^tQ_s^{-1}(\bd^*R+\nabla \Ric)(Q_s(k(s)v), Q_s(v))\, ds\big| \r] \notag\\
&\leq  \e^{\frac{K_0t}{2}} \E\l[\1_{\{t\leq \tau_D\}} \big| \e^{-\frac{K_0t}{2}}Q_t\int_0^t Q_s^{-1} R(\ptr_s d B_s, Q_s(k(s)v))Q_s(v)\big|^2\r] ^{1/2} \notag\\
&\quad + \frac{K_2}{2} \E\l[\1_{\{t\leq \tau_D\}}\big|\e^{\frac{1}{2}K_0t}\int_0^t\e^{\frac{1}{2}K_0s}|k(s)|\, ds\big| \r]. 
\end{align}
Moreover, 
\begin{align*}
&  d \Big|\e^{-\frac{1}{2}K_0t }Q_t\int_0^t Q_s^{-1} R(\ptr_s d B_s, Q_s(k(s)v))Q_s(v)\Big|^2\\
  &=2\e^{-K_0t}\Big\langle R(\ptr_t d B_t,Q_t(k(t)v))Q_t(v), Q_t\int_0^t Q_s^{-1} R(\ptr_s d B_s, Q_s(k(s)v))Q_s(v)\Big\rangle\\
&\quad+\e^{-K_0t}\big|R^{\sharp,\sharp}(Q_t(k(t)v), Q_t(v))\big |_{\rm HS}^2\, d t\\
&\quad-\e^{-K_0t}\Ric\l(Q_t\int_0^t Q_s^{-1} R(\ptr_s d B_s, Q_s(k(s)v))Q_s(v), Q_t\int_0^t Q_s^{-1} R(\ptr_s d B_s, Q_s(k(s)v))Q_s(v)\r)\,d t \\
&\quad-K_0\e^{-K_0t}\big|Q_t\int_0^t Q_s^{-1} R(\ptr_s d B_s, Q_s(k(s)v))Q_s(v)\big|^2\, dt \\
& \overset{m}{\leq } \e^{-K_0t}\big|R^{\sharp,\sharp}(Q_t(k(t)v),Q_t(v))\big|_{\rm HS}^2\, d t\leq K_1^2\e^{-K_0t}|Q_t|^4 k(t)^2 \, d t\leq K_1^2\e^{K_0t}k(t)^2 \, d t, \qquad    t\leq \tau_D.
\end{align*}
Combining this with  \eqref{L1-Wt},  we have 
\begin{align*}
\E\l(|W_t^k(v,v)| \1_{\{t\leq \tau_D\}}\r) 
& \leq K_1 \e^{\frac{1}{2}K_0t} \l(\int_0^t \e^{K_0s}k(s)^2\, ds\r)^{1/2} +\frac{K_2}{2} \e^{K_0t}  \int_0^t|k(s)|\,ds.
\end{align*}
We then complete the proof.
\end{proof}

By the results above, the following Hessian formula for eigenfunctions
$\phi$ is obtained.

\begin{theorem}\label{th0}
Let $D$ be a compact Riemannian manifold with
  boundary $\partial D$. Let $X{\bf.}(x)$ be a Brownian motion starting
  from $x\in D$ and $\tau_D$ be its first hitting time of $\partial D$.
Suppose that $k$ is a  non-negative function in $C^1_\text{b}([0,\infty);\R)$
such that $k(0)=1$. Then for $(\phi,\lambda)\in \Eig(\Delta)$, $t\geq 0$ and
$v\in T_xD$, 
\begin{align}\label{Hessian-formula-local_0}
(\Hess \phi)(v,v)&=\E^x\l[\e^{(t\wedge \tau_D)\lambda/2}(\Hess \,\phi)(Q_{t\wedge \tau_D}(k(t\wedge \tau_D)v), Q_{t\wedge \tau_D}(v))+\e^{(t\wedge \tau_D)\lambda/2}({\bf d} \phi)
(W^k_{t\wedge \tau_D}(v,\,v))\r] \notag\\
&\quad-\E^x\l[\e^{(t\wedge \tau_D)\lambda/2}\bd \phi(Q_{t\wedge \tau_D}(v))\int_0^{t\wedge \tau_D}\langle Q_s(\dot{k}(s)v),
\ptr_sd  B_s \rangle\r].
\end{align}
\end{theorem}

\begin{proof}
  
The claim follows by taking expectation of the martingale \eqref{martingale-1}
at time $0$ and $t\wedge \tau_D$. Recall that $|Q_t|\leq \e^{K_0t/2}$.
For $x\in\partial D$ formula \eqref{Hessian-formula-local_0} is obviously
tautological since $\tau_D\equiv 0$.
\end{proof}

To derive Hessian estimates of $\phi$ from Theorem \ref{th0} requires estimates of
$\Hess\phi$ on the boundary $\partial D$.
To this end, we first note the following observation. Since $\phi=0$ on the boundary $\partial D$, 
we have $\nabla\phi=N(\phi)N$.

\begin{lemma}\label{lem0}
For $x\in \partial D$ let $H(x)$ be the mean curvature of the boundary. Then
\begin{align*}
N^2(\phi)(x)=-H(x)N(\phi)(x),\quad x\in \partial D.
\end{align*}
\end{lemma}

\begin{proof}
  For $x\in \partial D$, we have
  \begin{align*}
    0&=\lambda \phi(x)=\Delta \phi(x)\\
     &=\div (\nabla \phi)(x)=\div(N(\phi)N)(x)\\
     &=\langle \nabla N(\phi), N \rangle(x)+N(\phi)\div(N)(x).
  \end{align*}
  Taking into account that $\div(N)(x)=H(x)$, the proof is completed.
\end{proof}

The following lemma is taken from \cite[Proposition 2.5]{ATW} and
allows to estimate the values of $|\nabla \phi|$ on the boundary.

\begin{lemma}\label{lem1}
  Let $\alpha_0\in \mathbb{R}$ such that
  \begin{align}\label{Deltarho}
   \Delta\rho_{\partial }\leq \alpha_0
  \end{align}
  outside ${\rm Cut}(\partial D)$. 
  Then for any $t>0$,
  \begin{align*}
   \|\nabla \phi\|_{\partial D, \infty}
    =\|N(\phi)\|_{\partial D,\infty}\leq \|\phi\|_{\infty}\e^{\lambda
      t/2}\left(\alpha_0^++\frac{\sqrt{2}}{\sqrt{\pi t}}\right).
      \end{align*}
   In particular, 
   \begin{align}\label{esti-Nphi}
   \|\nabla \phi\|_{\partial D, \infty}
   \leq \|\phi\|_{\infty}\e^{1/2}\left(\alpha_0^++\frac{\sqrt{2\lambda}}{\sqrt{\pi }}\right).
      \end{align}
    \end{lemma}

\begin{remark}
   With constants $K_0,\theta>0$ such that $\Ric\geq -K_0$ on $D$ and
  $H\geq -\theta$ on the boundary $\partial D$, where $H(x)$ is the
  mean curvature of $D$ at $x\in D$, let
  \begin{align*}
    \alpha_0=\max\l\{\theta, \sqrt{(n-1)K_0}\,\r\}.
  \end{align*}
  Then estimate \eqref{Deltarho} holds true for this $\alpha_0$.
 \end{remark}

Next, we introduce some results on local time estimate of reflecting Brownian motion, which is also  a tool in the boundary estimate of $|\Hess \phi|$. Let us recall some basic notations on it.
The reflecting Brownian motion on $D$ with generator $\frac12\Delta$ satisfies the SDE
$$d X_t=\ptr_t\circ\,d B_t^x+\frac12N(X_t)\, d l_t,\quad X_0=x,$$
where $B_t^x$ is a standard Brownian motion on the Euclidean space $T_xD\cong \R^n$ and $l_t$ is
the local time supported on $\partial D$ (see \cite{Wbook2} for details).
Now we turn to the problem of estimating $\E[\e^{\alpha l_t/2}]$ for $\alpha>0$ by exploiting a specific class of functions $h$. 

\begin{lemma}\label{esti:local-time}
Suppose that $h\in C^{\infty}(D)$ such that $h\geq 1$ and $N\log h \geq 1$. For $\alpha>0$ let
$$K_{h,\alpha}=\sup\big\{-\Delta\log h+\alpha |\nabla \log h|^2\big\}.$$
 Then 
\begin{align*}
\mathbb{E}[\e^{\alpha l_t/2}]\leq \|h\|_{\infty}^{\alpha}\exp\l(\frac{\alpha}{2}K_{h,\alpha} t\r).
\end{align*}
\end{lemma}

\begin{proof}
By the It\^{o} formula we have
\begin{align*}
d h^{-\alpha}(X_t)&=\langle \nabla h^{-\alpha}(X_t), \ptr_t \, d B_t \rangle
+\frac12\Delta h^{-\alpha}(X_t)\,d t +\frac12Nh^{-\alpha}(X_t)\,d l_t\\
&\leq \langle \nabla h^{-\alpha}(X_t), \ptr_t \, d B_t \rangle-\alpha h^{-\alpha}(X_t)
\l(-\frac12K_{h,\alpha} \,dt+\frac12N\log h(X_t)\,dl_t\r).
\end{align*} 
Hence,
\begin{align*}
M_t:=h^{-\alpha}(X_t) \exp{\l(-\frac{\alpha}{2} K_{h,\alpha} t+\frac{\alpha}{2} \int_0^t N\log h(X_s)\,dl_s\r)}
\end{align*}
is a local submartingale.  Therefore, by Fatou's lemma and taking into account that $h\geq 1$, we get
\begin{align*}
  &\E\l[h^{-\alpha}(X_t)\exp\l(-\frac{\alpha}{2} K_{h,\alpha}t+\frac{\alpha}{2}
   \int_0^t N \log h(X_s)\, dl_s\r)\r]\\
  & \leq \E\l[h^{-\alpha}(X_{t\wedge \tau_D})
   \exp \l(-\frac{\alpha}{2} K_{h,\alpha}(t\wedge \tau_D)+\frac{\alpha}{2}
   \int_0^{t\wedge \tau_D} N\log h(X_s)\,dl_s\r)\r]\\
&\leq h^{-\alpha}(x)\leq 1.
\end{align*}
Since $N\log h(x)\geq 1$ we conclude that
\begin{equation*}
\E \l[\exp \Big(\frac{\alpha}{2} l_t\Big)\r]\leq \E\l[\exp \l(\frac{\alpha}{2} \int_0^t N\log h(X_s)\,dl_s\r)\r]\leq \|h\|_{\infty}^{\alpha}\exp\Big(\frac{\alpha}{2}K_{h,\alpha} t\Big).\qedhere 
\end{equation*}
\end{proof}

At the end of this subsection, we introduce some results on  Hessian comparison of $\rho_{\partial}$. 
Let $p$ be the orthogonal projection of $x$ on $\partial D$, and 
let $\gamma (s)=\exp_p(sN), s\in [0,\rho_{\partial} (x)]$ be the geodesic 
from $p$ to $x$. Let $\{J(s)\}_{s\in [0,\rho_{\partial}(x)]}$ be the Jacobi
field along $\gamma$ such that $J(\rho_{\partial}(x))=v$ for $v\in T_{x}D$, and
$\dot{J}(0)=-\II^{\sharp}(J(0))\in T_p\partial D$, where 
$\langle \II ^{\sharp}(J(0)), w\rangle= \II(J(0), w)$ for $w\in T_p\partial D$. 
From the variation formula of $\rho_{\partial}$, we know that 
  \begin{align}\label{variation-formula}
  \Hess \rho_{\partial} (v,v)=-\II(J(0),J(0))+\int_0^{\rho_\partial(x)}
  \l(|\dot{J}(s)|^2-\langle R(\dot{\gamma}(s), J(s))\dot{\gamma}(s), J(s) \rangle\r)\,ds.
  \end{align}
  The following result is essentially due to Kasue \cite{Kasue82,Kasue84} (see also Theorem A.1 in \cite{Wa05}).
  \begin{lemma}[Hessian Comparison]\label{Hessian Comparison}
  Let $\sigma$ and $k$ be non-negative constants such that $|\II|\leq \sigma$ and $|\Sect|\leq k$ on $\partial_{r_0}D$, where $\rho_{\partial}$ is smooth $\partial_{r_0}D$. Then
  \begin{align*}
 \frac{\ell'_{\sigma, k}}{\ell_{\sigma, k}}(\rho_{\partial}(x)) \leq \Hess {{\rho}_\partial}(v,v)\leq \frac{\ell'_{-\sigma, -k}}{\ell_{-\sigma, -k}}(\rho_{\partial}(x)), \quad  \rho_{\partial}\leq r_0\wedge \ell_{\sigma, k}^{-1}(0).
  \end{align*}
 Moreover, for $\rho_{\partial}(x)\leq r_0\wedge \ell_{\sigma, k}^{-1}(\frac{1}{2})$,
 \begin{align*}
 |\Hess {{\rho}_\partial}|\leq 2\max\{\sigma, \sqrt{k} \} .
 \end{align*} 
  \end{lemma}
\begin{proof}
The proof of  first inequality can be found in \cite[Theorem 1.2.2]{Wbook2}. Based on this, 
it is easy to have for $k,\sigma\geq 0$, 
\begin{align*}
\Hess \rho_{\partial}(v,v)\leq  \max \{\sigma, \sqrt{k}\}.
\end{align*}
For $\rho_{\partial}(x)\leq r_0\wedge \ell_{k,\sigma}^{-1}(\frac{1}{2})$,
\begin{align*}
\Hess \rho_{\partial}(v,v)\geq \frac{\ell'_{k,\sigma}(\rho_{\partial} (x))}{\ell_{k,\sigma}(\rho_{\partial} (x))}\geq 2 \ell_{k,\sigma}(\rho_{\partial} (x))\geq -2\max\{\sigma, \sqrt{k}\}.
\end{align*}
We then complete proof of the second inequality.
\end{proof}

\subsection{Hessian estimate of Dirichlet eigenfunctions}

Lemmas \ref{lem0}, \ref{lem1} and \ref{esti:local-time} allow to derive an estimate of $|\Hess \phi|$ on the boundary
$\partial D$.

\begin{lemma}\label{lem3}
 Let $K_0, \sigma $ be non-negative constants such that  $\Ric\geq -K_0$, $|\II|\leq \sigma$. Suppose that
  the distance function  $\rho_{\partial}$ is  smooth on $\partial_{r_0}D:=\{x: \rho_{\partial }(x)\leq r_0\}$ for some constant $r_0>0$. 
  Then for $x\in \partial D$,
  \begin{align*}
   \big\|\Hess(\phi)\big\|_{\partial D,\infty} &\leq 
  (n-1)\sigma \,\|N(\phi)\|_{\partial D, \infty}\\
                                               &\quad+\|h\|_{\infty}^{\sigma}\e^{\frac{1}{2}(K_0+\sigma K_{h,\sigma})t}\l(C_1\frac{1}{\sqrt{t}}+C_2\sqrt{t}\r)\|\phi\|_{\infty}\\
 &\quad+  \|h\|_{\infty}^{\sigma}\e^{\frac{1}{2}(K_0+\sigma K_{h,\sigma})t}\l(\frac{1}{\sqrt{t}}+C_3\sqrt{t}\r)\|\nabla \phi\|_{\infty}\\
 &\quad +\|h\|_{\infty}^{\sigma}\e^{\frac{1}{2}(K_0+\sigma K_{h,\sigma})t}\sqrt{t}C_4\|\Hess \phi\|_{\infty}
  \end{align*}
  where $h\in C^{\infty}(D)$ such that $h\geq 1$ and $N\log h \geq 1$ and
$$K_{h,\sigma}=\sup\{-\Delta\log h+\sigma |\nabla \log h|^2\},$$
and the  constant $C_1,C_2, C_3, C_4$ are defined as
\begin{align*}
&C_1= \|\Delta \rho_{\partial} \|_{\partial_{r_0}D}, \\
& C_2 = \| \Delta (\psi (\rho_{\partial}))  \Delta \rho_{\partial}+2 \psi'(\rho_{\partial})|\nabla(\Delta \rho_{\partial})|+\psi(\rho_{\partial})(\lambda \Delta \rho_{\partial}+\Delta^2 \rho_{\partial})\|_{\partial_{r_0}D},\\
&C_3=\|\Delta (\psi (\rho_{\partial}))+2 \psi'(\rho_{\partial}) \Delta \rho_{\partial} +\psi(\rho_{\partial}) (3|\nabla(\Delta \rho_{\partial})|+\lambda)\,\|_{\partial_{r_0}D},\\
&C_4= \|2\psi'(\rho_{\partial})+2(n-1)\psi(\rho_{\partial})|\Hess \rho_{\partial}|\,\|_{\partial_{r_0}D}.
 \end{align*}
 where $\psi\in C^2(\R^+, [0,1])$ satisfies $\psi(0)=1$, $\psi'(0)=0$ and $\psi(r)=0$ for $r>r_0$.
\end{lemma}

\begin{proof}
  Given $x\in \partial D$, let $\{X_i\}_{1\leq i\leq n}$ be an
  orthonormal basis of $T_xD$ with $X_1=N$. Then
  \begin{align*}
    |\Hess(\phi)(X_i,X_j)|
    &=|\nabla \bd \phi(X_i,X_j)|=|\langle\nabla_{X_i}\nabla \phi,X_j\rangle|\\
    &=|X_i\langle \nabla \phi, X_j \rangle-\langle \nabla \phi, \nabla_{X_i} X_j\rangle|.
  \end{align*}
 By assumption we have $|\II|\leq \sigma$.
  If $X_i,\,X_j \in T_x \partial D$, i.e.~$i,j \neq 1$, then $\langle \nabla \phi, X_j\rangle|_{\partial D}=0$ and
  \begin{align}\label{add-Hess-boundary1}
    |\Hess(\phi)(X_i,X_j)|=|-N(\phi)\langle N, \nabla_{X_i} X_j\rangle|\leq \sigma|N(\phi)|.
  \end{align}
  If $X_i=X_j=N$, i.e. $i=j=1$, then $\nabla _N N|_{\partial D}=0$ and
  \begin{align}\label{add-Hess-boundary2}
    |\Hess(\phi)(N,N)|= |N^2(\phi)|\leq |H N(\phi)|\leq (n-1)\sigma |N(\phi)|. 
  \end{align}
  If $X_j\in T_x\partial D$ and $X_i=N$ (i.e. $j\neq 1$ and $i=1$),
  then
  \begin{equation}\label{HessXN}
    |\Hess(\phi)(X_j,N)|(x)=|NX_j(\phi)|(x).
  \end{equation}
  In order to get control on \eqref{HessXN}, we shall use a probabilistic argument based
  on the Brownian motion
  on $D$ reflected at the boundary. Before going into the details, we make a general remark
  on the extension of vector fields from $\partial D$ to a tubular neighborhood of the boundary.

\begin{remark}
Assuming that the boundary $\partial D$ is smooth, let $N$ be the unit inward normal vector field $N$ on $\partial D$.
Furthermore, let
\begin{align}\label{TubularNH}
  \Phi\colon {[0,r_0[}\times \partial D\to D,\quad (r,x)\mapsto \exp_x(rN),
\end{align} 
be the geodesic from $x\in\partial D$ orthogonal to $\partial D$ and parametrized by its arc length $r$.
As the differential of $\Phi$ at any point $(0,x)$ has full rank, we find $\varepsilon_0>0$ such that
$\Phi$ is a diffeomorphism from ${[0,\varepsilon_0[}\times \partial D$ 
onto the open neighborhood $\{x\in D\colon \rho_{\partial}(x)<\varepsilon_0\}$ of $\partial D$ in $D$.
This allows to extend $N$ to a tubular (collar) neighborhood of $\partial D$ as
$\Phi_*\frac{\partial}{\partial r}$.
By construction then $\nabla_NN=0$.
If $X$ is a vector field on $\partial D$ tangential to $\partial D$, we extend it to the neighborhood of $\partial D$ as being independent of the real variable in the product 
${[0,\varepsilon_0[}\times \partial D$. By construction, close to the boundary, the distance function
$\rho_{\partial}(x)=\text{dist}(x,\partial D)$ is smooth and satisfies $N=\nabla\rho_{\partial }$.
\end{remark}

 Let $N$ be the extension of the normal vector field to a tubular neighborhood $\partial_{r_0}D:=\{x: \rho_{\partial }(x)\leq r_0\}$ of $\partial D$ and define
 \begin{align}\label{def-varphi}
 \varphi(x)=\psi(\rho_{\partial}(x))\div (\phi N),\qquad x\in \partial_{r_0} D,
 \end{align}
 where $\psi\in C^2(\R^+, [0,1])$ satisfies $\psi(0)=1$, $\psi'(0)=0$ and $\psi(r)=0$ for $r>r_0$.
 Using the formula
$\div(\phi N)=N(\phi)+\phi\,\div(N)$, along with Lemma \ref{lem0}, we observe for $x\in \partial D$,
$$N(\varphi)(x)=\psi'(0)\div (\phi N)+ N(\div (\phi N))=0.$$
Thus $\varphi$ satisfies the Neumann boundary conditions on $D$.

  Let now $X_t$ be the reflecting Brownian motion on $D$ and  $P_{t}^Nf(x)=\E^x[f(X_t)]$ for $f\in \mathcal{B}_b(D)$ the corresponding Neumann semigroup. 
 According to the Kolmogorov equation,
 \begin{align*}
\varphi(x)&=P_{t}^{N}( \varphi)(x)-\frac{1}{2}\int_0^tP_s^{N} (\Delta \varphi)(x)\,ds.
 \end{align*}
 Taking derivative on both sides of the above equation yields
 \begin{align*}
X_i( \varphi)(x)= X_i (P_{t}^{N} \varphi)(x)-\frac{1}{2}\int_0^t
X_i ( P^N_{s}\Delta  \varphi)(x)\, ds
 \end{align*}
where $X_i$ is tangential to $\partial D$. We first observe that 
 for $x\in \partial D$,
 \begin{align*}
  X_i( \varphi)(x)&=X_i(\psi (\rho_{\partial}))(x)\div (\phi N)(x)+\psi(\rho_{\partial}(x))X_i(\div (\phi N))(x)=X_i(\div (\phi N))(x)\\
  &=X_iN(\phi)(x)+X_i(\phi)(x)\div (N)(x)+\phi(x)X_i(\div (N))(x)\\
  &=X_iN(\phi)(x).
  \end{align*}
  To deal with the upper bound, we use the Bismut formula established in \cite[Theorem 3.2.1]{Wbook2} for the
  compact manifold $D$, which gives 
 \begin{align*}
 |\nabla P_t^{N}f|\leq \frac{1}{\sqrt{t}}\e^{\frac{1}{2}K_0t} \E^x[\e^{\sigma l_t}]^{\frac{1}{2}}\|f\|_{\infty},
 \end{align*}
 where $l_t$ is the local time supported on $\partial D$. By 
 Lemma \ref{esti:local-time} derived in the previous subsection, we have
 \begin{align*}
 \E^x[\e^{\sigma l_t}]\leq \|h\|_{\infty}^{2\sigma} \exp\l(\sigma K_{h,2\sigma} t\r),
 \end{align*}
 where $h\in C^{\infty}(D)$ such that $h\geq 1$ and $N\log h \geq 1$ and
$$K_{h,2\sigma}=\sup\{-\Delta\log h+2\sigma |\nabla \log h|^2\}.$$
We then conclude that
 \begin{align}\label{boundaryHessian-1} 
 |X_iN(\phi)|(x)&\leq \|h\|_{\infty}^{\sigma}\e^{\frac{1}{2}(K_0+\sigma K_{h,2\sigma})t}\l[\frac{1}{\sqrt{t}}\,\|\varphi\|_{B(x,r_0)}+ \sqrt{t} \,\|\Delta \varphi\|_{B(x, r_0)}   \r].
 \end{align} 
According to the definition of $\varphi$ in \eqref{def-varphi}, we have
 \begin{align*}
\|\varphi\|_{\infty} \leq \| \nabla \phi\|_{\infty}+\|\div(N)\|_{\partial_{r_0}D} \|\phi\|_{\infty}
 \end{align*}
 By commutation rules, we calculate 
 \begin{align}\label{delta-est1}
 \Delta ((\psi(\rho_{\partial})) \div (\phi N))=& \Delta (\psi (\rho_{\partial})) \div (\phi N))
 +2\psi'(\rho_{\partial}) N(\div(\phi N))+\psi (\rho_{\partial})) \Delta (\div (\phi N))\notag\\
 =& \Delta (\psi (\rho_{\partial})) (\phi \div(N)+N(\phi))
 +2\psi'(\rho_{\partial}) \Big(\phi N(\div(N))+N(\phi) \div(N)+N^2(\phi)\Big)\notag\\
 &+\psi (\rho_{\partial})) \Delta (\div (\phi N))
 \end{align}
 and 
 \begin{align}\label{delta-est2}
 \Delta (\div (\phi N))&= \div ((\Box-\Ric^{\sharp})(\phi N) )\notag\\
 &= \div(\Delta (\phi) N)+ \div(\phi\Box N)+2\div(\nabla _{\nabla \phi}N)-\phi \div (\Ric^{\sharp}(N))-\Ric(N,\nabla\phi)\notag\\
 &= -\lambda \div(\phi N)+\phi \div((\Box-\Ric^{\sharp}) N)+\langle \Box N, \nabla \phi \rangle+2\div(\nabla _{\nabla \phi}N)-\Ric(N,\nabla\phi),
 \end{align}
 where $\Box =\tr \nabla^2$ and $\Ric^{\sharp}\colon TD\rightarrow TD$ such that $\langle \Ric^{\sharp}(v), w \rangle=\Ric(v,w)$ for $v,w\in T_xD$, $x\in D$. 
 Let $\{e_i\}_{1\leq i\leq n}$ be orthonormal basis of $TD$ around $x$ satisfying $\nabla e_i(x)=0$. 
 We then have
 \begin{align*}
 \nabla _{\nabla \phi}N=\sum_{i=1}^n (e_i(\phi)) \nabla _{e_i}N,
 \end{align*}
 and as a consequence
 \begin{align*}
 \div (\nabla _{\nabla \phi} N)&=\sum_{i=1}^n\l[ \langle \nabla e_i(\phi), \nabla_{e_i}N\rangle +e_i(\phi) \div(\nabla _{e_i}N)\r]\\
 &=\langle \Hess_{\phi}, \nabla N \rangle+\langle \nabla \phi, \sum_{i=1}^n \div(\nabla_{e_i}N)e_i \rangle\\
 &=\langle \Hess_{\phi}, \nabla N \rangle+\langle \nabla \phi, \nabla (\div (N)) \rangle.
 \end{align*}
 Combining this with \eqref{delta-est2} yields
 \begin{align*}
 \Delta (\div (\phi N))
 &=\phi(-\lambda \div(N)+\Delta(\div(N)))-\lambda N(\phi)+2\langle \Hess (\phi), \nabla N \rangle+2\langle\nabla \phi, \nabla (\div(N)) \rangle \notag\\
 &\quad+\langle \Box N, \nabla \phi \rangle-\Ric(N,\nabla\phi).
 \end{align*}
 From the fact that $N=\nabla \rho_{\partial}$ and the Weitzenb\"{o}ck formula, we observe that
 \begin{align}\label{delta-est3}
 \div(N)=\Delta \rho_{\partial},\ \  \nabla N=\Hess \rho_{\partial},\ \  \mbox{and} \quad   \langle \Box N, \nabla \phi \rangle)-\Ric(N,\nabla\phi)=\langle \nabla \Delta \rho_{\partial}, \nabla \phi \rangle.
 \end{align}
 Combining the equations \eqref{delta-est1}, \eqref{delta-est2} and \eqref{delta-est3} with \eqref{boundaryHessian-1},  we finally conclude that
 \begin{align*}
& |X_iN(\phi)|(x)\\
&\leq 
  \|h\|_{\infty}^{\sigma}\e^{\frac{1}{2}(K_0+2\sigma K_{h,\sigma})t}\l(C_1\frac{1}{\sqrt{t}}+C_2\sqrt{t}\r)\|\phi\|_{\infty}+  \|h\|_{\infty}^{\sigma}\e^{\frac{1}{2}(K_0+\sigma K_{h,2\sigma})t}\l(\frac{1}{\sqrt{t}}+C_3\sqrt{t}\r)\|\nabla \phi\|_{\infty}\\
 &\quad  +\|h\|_{\infty}^{\sigma}\e^{\frac{1}{2}(K_0+\sigma K_{h,2\sigma})t}\sqrt{t}\,C_4\|\Hess \phi\|_{\infty}
 \end{align*}
 where 
 \begin{align*}
&C_1= \|\Delta \rho_{\partial} \|_{\partial_{r_0}D}, \\
& C_2 = \| \Delta (\psi (\rho_{\partial}))  \Delta \rho_{\partial}+2 \psi'(\rho_{\partial})|\nabla(\Delta \rho_{\partial})|+\psi(\rho_{\partial})(\lambda \Delta \rho_{\partial}+\Delta^2 \rho_{\partial})\|_{\partial_{r_0}D},\\
&C_3=\|\Delta (\psi (\rho_{\partial}))+2 \psi'(\rho_{\partial}) \Delta \rho_{\partial} +\psi(\rho_{\partial}) (3|\nabla(\Delta \rho_{\partial})|+\lambda)\,\|_{\partial_{r_0}D},\\
&C_4= \|2\psi'(\rho_{\partial})+2(n-1)\psi(\rho_{\partial})|\Hess \rho_{\partial}|\,\|_{\partial_{r_0}D}.
 \end{align*}
 The proof is completed by combining the above estimate with \eqref{add-Hess-boundary1} and
 \eqref{add-Hess-boundary2}.
 \end{proof}

Combining the estimates in Lemmas \ref{lem1} and \ref{lem3}  with
Theorem \ref{th0}, we are now in a position to prove our  main
result.

\begin{theorem}\label{main-th0}
Let $D$ be a compact Riemannian manifold with
boundary $\partial D$.
Let   $K_0,K_1$, $K_2$ and $\sigma$ be non-negative constants such that  $\Ric\geq -K_0$, $|R|\leq K_1$ and $|\bd^* R+\nabla \Ric|\leq K_2$ on $D$, and that   $|\II|\leq \sigma$ on the boundary $\partial D$. Assume  the distance function  $\rho_{\partial}$ is smooth on the tubular neighborhood $\partial_{r_0}D:=\{x: \rho_{\partial }(x)\leq r_0\}$ of
$\partial D$ for some constant $r_0>0$, and let $\alpha, \beta, \gamma \in \mathbb{R}$ be such that
  \begin{align}\label{condition-rho1}
  |\Hess \rho_{\partial }|\leq \frac{\alpha}{n-1},\quad |\nabla (\Delta \rho_{\partial })|\leq \beta,\quad |\Delta^2 \rho_{\partial}|\leq \gamma \quad 
    \mbox{on}\ \partial_{r_0}D.
  \end{align}
  For $h\in C^{\infty}(D)$ with
$\min_Dh =1$ and $N\log h|_{\partial D}\geq 1$, 
then 
\begin{align}\label{add-Hess-D2}
\frac{\|\Hess \phi\|}{\|\phi\|_{\infty}}
&\leq  2(n-1)\e \sigma  \l(\alpha+\sqrt{\frac{2\lambda}{\pi}}\r)+2\alpha\|h\|_{\infty}^{\sigma}\sqrt{\e} \max \l\{\sqrt{\lambda+2K_0+\sigma K_{h,2\sigma}},\, 2\sqrt{\e}\|h\|_{\infty}^{\sigma}\left( \frac{6}{r_0}+2\alpha\right)\r\} \notag\\
&\quad + \frac{\frac{3}{r_0}(\alpha^2+2\beta)+ \frac{6}{r_0^2}\alpha +2\lambda \alpha+\gamma}{ \frac{6}{r_0}+2\alpha}\notag\\
&\quad+\l(\frac{\frac{9}{r_0} \alpha + \frac{6}{r_0^2}  + 3\beta +\lambda +K_1}{ \frac{6}{r_0}+2\alpha}  +\frac{K_2}{4\sqrt{\e}\left( \frac{6}{r_0}+2\alpha\right)^2}\r) \sqrt{\e}\l(\alpha +\l(\frac{1}{4}\sqrt{\frac{\pi}{2}}+\sqrt{\frac{2}{\pi}}\r)\sqrt{\lambda+K_0}\r) \notag\\
  &\quad + 4{\e}\,\|h\|_{\infty}^{\sigma}  \max \l\{\sqrt{\lambda+2K_0+\sigma K_{h,2\sigma}}\,,\, 2\sqrt{\e}\,\|h\|_{\infty}^{\sigma}\left( \frac{6}{r_0}+2\alpha\right)\r\}\notag\\
  &\qquad\qquad\qquad\qquad\qquad\times\l(\alpha +\l(\frac{1}{4}\sqrt{\frac{\pi}{2}}+\sqrt{\frac{2}{\pi}}\r)\sqrt{\lambda+K_0}\r).
\end{align}

\end{theorem}

\begin{proof}
According to formula \eqref{Hessian-formula-local_0} we have
\begin{align*}
|\Hess\phi (v,v)|&= \E\Big[\e^{\lambda (t\wedge \tau_D)/2} \Hess \phi \big(Q_{t\wedge \tau_D}(k(t\wedge \tau_D)v), Q_{t\wedge \tau_D}(v)\big)\Big]\\
&\quad +\E\Big[\e^{\lambda (t\wedge \tau_D)/2} \bd \phi
      (W_{t\wedge \tau_D}^k(v, v))\Big]\\
  &\quad     -\E\l[\e^{\lambda (t\wedge \tau_D)/2} \bd \phi(Q_{t\wedge \tau_D}( v))\int_0^{t\wedge \tau_D}\langle Q_s(\dot{k}(s)v),
     \ptr_s d  B_s \rangle\r].
\end{align*}
Taking $k(s)=(t-s)/{t}$ for $s\in [0,t]$ in the equation yields
\begin{align*}
|\Hess \phi (v,v)| 
&\leq  \E\l[\1_{\{\tau_D\leq t\}}\e^{(\frac{\lambda}{2}+K_0)\tau_D}\frac{t-\tau_D}{t}\|\Hess (\phi)\|_{\partial D,\infty}\r]  \\
  &\quad + \|\bd \phi\|_{\infty} \l(K_1\sqrt{t}+\frac{K_2}{2}t\r)
    \,\e^{\big(\frac12\lambda+K_0\big)t}\\
&\quad + \|\bd \phi\|_{\infty}\frac{\e^{\big(\frac12\lambda+K_0\big) t}}{\sqrt{t}}.
\end{align*}
By Lemmas \ref{lem1} and \ref{lem3}, we have
\begin{align}\label{eq-Hess1}
|\Hess \phi (v,v)|
&\leq \E\Bigg\{\1_{\{\tau_D\leq t\}}\e^{(\frac{\lambda}{2}+K_0)\tau_D}\frac{t-\tau_D}{t}\Bigg[\max\big\{ \|H\|_{\partial D, \infty},\sigma\big\}\|N\phi\|_{\partial D, \infty}\notag\\
  &\quad +\|h\|_{\infty}^{\sigma}\e^{\frac{1}{2}(K_0+\sigma K_{h,2\sigma})(t-\tau_D)}\l(C_1\frac{1}{\sqrt{t-\tau_D}}+C_2\sqrt{t-\tau_D}\r)\|\phi\|_{\infty}\notag\\
  &\quad  +  \|h\|_{\infty}^{\sigma}\e^{\frac{1}{2}(K_0+\sigma K_{h,2\sigma})(t-\tau_D)}\l(\frac{1}{\sqrt{t-\tau_D}}+C_3\sqrt{t-\tau_D}\r)\|\bd \phi\|_{\infty}\notag\\
 &\quad   +\|h\|_{\infty}^{\sigma}\e^{\frac{1}{2}(K_0+\sigma K_{h,2\sigma})(t-\tau_D)}\sqrt{t-\tau_D}C_4\|\Hess \phi\|_{\infty}\Bigg]\Bigg\}  \notag \\
  &\quad + \|\bd \phi\|_{\infty} \l(K_1\sqrt{t}+\frac{K_2}{2}t\r)
    \,\e^{\big(\frac12\lambda+K_0\big)t}\notag\\
&\quad + \|\bd \phi\|_{\infty}\frac{\e^{\big(\frac12\lambda+K_0\big) t}}{\sqrt{t}},
\end{align}
where $C_1, C_2, C_3$ and $C_4$ are defined as in Lemma \ref{lem3}.
Combining this with
the fact that
\begin{align*}
\frac{t-\tau_D}{t}\frac{1}{\sqrt{t-\tau_D}}= \frac{\sqrt{t-\tau_D}}{t}\leq \frac{1}{\sqrt{t}}
\end{align*}
 and then substituting back into \eqref{eq-Hess1} and using \eqref{esti-Nphi},  we obtain 
\begin{align}\label{ineqn-Hess2}
  |\Hess \phi (v,v)|&\leq  (n-1)\sigma \e^{\big(\frac{\lambda}{2}+K_0\big)t} \sqrt{\e}  \l(\alpha+\sqrt{\frac{2\lambda}{\pi}}\r)\|\phi\|_{\infty} \notag\\
   & \quad  + \|h\|_{\infty}^{\sigma}\e^{\big(\frac{\lambda}{2}+K_0+\frac{\sigma K_{h,2\sigma}}{2}\big)t}\l(\frac{C_1}{\sqrt{t}}+ C_2\sqrt{t}\r)\|\phi\|_{\infty} \notag \\
   & \quad  + \|h\|_{\infty}^{\sigma}\e^{\big(\frac{\lambda}{2}+K_0+\frac{\sigma K_{h,2\sigma}}{2}\big)t}\l(\frac{1}{\sqrt{t}}+ C_3\sqrt{t}\r)\|\bd \phi\|_{\infty} \notag \\
 & \quad   +C_4\|h\|^{\sigma}_{\infty}\e^{\big(\frac{\lambda}{2}+K_0+\frac{\sigma K_{h,2\sigma}}{2}\big)t}\sqrt{t}\,\|\Hess \phi\|_{\infty} \notag \\
  &\quad + \l(\frac{1}{\sqrt{t}}+K_1\sqrt{t}+\frac{K_2}{2}t\r)
    \,\e^{\big(\frac12\lambda+K_0\big)t} \|\bd \phi\|_{\infty}.
\end{align}
Now let $$t=t_0:=\frac{1}{\max\big\{\lambda+2K_0+\sigma K_{h,2\sigma},\, 4\e\|h\|_{\infty}^{2\sigma}C_4^2\big\}}.$$ Then 
\begin{align*}
\|h\|_{\infty}^{\sigma}\e^{(\frac{\lambda}{2}+K_0+\frac{\sigma K_{h,2\sigma}}{2})t_0}\sqrt{t_0}\,C_4\|\Hess \phi\|_{\infty}\leq \frac{1}{2}\|\Hess \phi\|_{\infty}
\end{align*}
and then inequality \eqref{ineqn-Hess2} becomes
\begin{align}\label{Hess-phi1}
  |\Hess \phi (v,v)| 
&\leq 2(n-1)\sigma {\e} \l(\alpha+\sqrt{\frac{2\lambda}{\pi}}\r)\|\phi\|_{\infty}\notag\\
&\quad +2C_1\|h\|_{\infty}^{\sigma}\sqrt{\e} \max \l\{\sqrt{\lambda+2K_0+\sigma K_{h,2\sigma}},\, 2\sqrt{\e}\|h\|_{\infty}^{\sigma}C_4\r\}\|\phi\|_{\infty}\notag \\
&\quad + \frac{C_2}{C_4} \|\phi\|_{\infty}+\frac{C_3}{C_4} \|\bd \phi\|_{\infty} \notag\\
&\quad + 2\sqrt{\e}(\|h\|_{\infty}^{\sigma} +1) \max \l\{\sqrt{\lambda+2K_0+\sigma K_{h,2\sigma}},\, 2\sqrt{\e}\|h\|_{\infty}^{\sigma}C_4\r\} \|\bd \phi\|_{\infty}\notag\\
&\quad + \frac{2 K_1\sqrt{\e}}{\max\l\{\sqrt{\lambda+2K_0+\sigma K_{h,2\sigma}},\, 2\sqrt{\e}\|h\|_{\infty}^{\sigma}C_4\r\}}  \|\bd \phi\|_{\infty}\notag\\
&\quad +\frac{K_2 \sqrt{\e} }{\max\l\{\lambda+2K_0+\sigma K_{h,2\sigma},\, 4\e\|h\|_{\infty}^{2\sigma}C_4^2\r\}} \|\bd \phi\|_{\infty}\notag\\
&\leq 2\alpha {\e} \l(\alpha+\sqrt{\frac{2\lambda}{\pi}}\r)\|\phi\|_{\infty}\notag\\
&\quad +2C_1\|h\|_{\infty}^{\sigma}\sqrt{\e} \max \l\{\sqrt{\lambda+2K_0+\sigma K_{h,2\sigma}},\, 2\sqrt{\e}\|h\|_{\infty}^{\sigma}C_4\r\}\|\phi\|_{\infty}+ \frac{C_2}{C_4} \|\phi\|_{\infty} \notag\\
&\quad + 4\sqrt{\e}\|h\|_{\infty}^{\sigma}  \max \l\{\sqrt{\lambda+2K_0+\sigma K_{h,2\sigma}},\, 2\sqrt{\e}\|h\|_{\infty}^{\sigma}C_4\r\} \|\bd \phi\|_{\infty}\notag\\
&\quad +\l(\frac{C_3}{C_4} +\frac{K_1}{C_4} +\frac{K_2}{4\sqrt{\e}C_4^2}\r) \|\bd \phi\|_{\infty}
\end{align}

It is known from  Arnaudon, Thalmaier and Wang \cite[Eq. (2.8)]{ATW} that
\begin{align*}
\frac{\|\bd \phi\|_{\infty}}{\|\phi\|_{\infty}}&\leq 
\sqrt{\e}\l(\alpha+\sqrt{\frac{2}{\pi}(\lambda+K_0)}+\frac{\lambda+K_0}{4(\alpha+\sqrt{\frac{2}{\pi}(\lambda+K_0)} )}\r)\\
&\leq 
\sqrt{\e}\l(\alpha+\l(\sqrt{\frac{2}{\pi}}+ \frac{1}{4}\sqrt{\frac{\pi}{2}}\r)\sqrt{\lambda+K_0}\r).
\end{align*}
Note here we use the upper bound $\alpha^++\sqrt{\frac{2}{\pi t}}$ of  $f(t,\alpha)$ defined in \cite{ATW} to simplify the upper bound in \cite[Eq. (2.8)]{ATW}.  Next, 
combining this with \eqref{Hess-phi1} implies that
\begin{align}\label{add-Hess-D3}
\frac{\|\Hess \phi\|}{\|\phi\|_{\infty}}
&\leq  2\e (n-1)\sigma\l(\alpha+\sqrt{\frac{2\lambda}{\pi}}\r)\notag\\
&\quad+2C_1\|h\|_{\infty}^{\sigma}\sqrt{\e} \max \l\{\sqrt{\lambda+2K_0+\sigma K_{h,2\sigma}},\, 2\sqrt{\e}\|h\|_{\infty}^{\sigma}C_4\r\}  + \frac{C_2}{C_4}\notag\\
&\quad+\l(\frac{C_3}{C_4} +\frac{K_1}{C_4} +\frac{K_2}{4\sqrt{\e}C_4^2}\r) \sqrt{\e}\l(\alpha+\l(\sqrt{\frac{2}{\pi}}+ \frac{1}{4}\sqrt{\frac{\pi}{2}}\r)\sqrt{\lambda+K_0}\r) \notag\\
&\quad + 4{\e}\|h\|_{\infty}^{\sigma} \max \l\{\sqrt{\lambda+2K_0+\sigma K_{h,2\sigma}},\, 2\sqrt{\e}\|h\|_{\infty}^{\sigma}C_4\r\}\l(\alpha+\l(\sqrt{\frac{2}{\pi}}+ \frac{1}{4}\sqrt{\frac{\pi}{2}}\r)\sqrt{\lambda+K_0}\r).
\end{align}
Using condition \ref{condition-rho1},  the constants 
 $C_1, C_2, C_3$ and $C_4$ now become
\begin{align*}
&C_1= \alpha, \\
& C_2 = \|\psi'\|_{\infty}(\alpha^2+2\beta)+ \|\psi''\|_{\infty}\alpha +(\lambda \alpha+\gamma),\\
&C_3=3\|\psi'\|_{\infty} \alpha + \|\psi''\|_{\infty}  +3\beta+\lambda,\\
&C_4= 2\|\psi'\|_{\infty}+2\alpha.
 \end{align*}
 Let
\begin{equation}\label{def-psi}
\psi(r)=\begin{cases}
\l(\frac{r_0-r}{r_0}\r)^3,\quad & 0\leq r\leq r_0;  \\
0,\quad &  r>r_0,
\end{cases}.
\end{equation} 
Then $\psi'\leq \frac{3}{r_0}$ and $\psi'' \leq \frac{6}{r_0^2}$. Form these estimates, the 
constants $C_1, C_2,C_3$ and  $C_4$ are further explicit.
\end{proof}

\subsection{Proof of Theorem \ref{main-theorem1}}\label{Section-Construction-of-h}
In this subsection we describe F.-Y. Wang's  construction of functions $h$ satisfying the requirements of Lemma~\ref{esti:local-time} (see
\cite[p.~1436]{W07} or \cite[Theorem 3.2.9]{Wbook2} for the details). His construction is performed under the following condition.
\medskip


\noindent{\bf Condition (A) }
There exist a non-negative constant $\sigma$ such that $\II \leq \sigma$ and a positive constant $r_1$ such that the distance function $\rho_{\partial }$
to the boundary $\partial D$ is smooth on $\partial_{r_1}D:=\{x\in D: \rho_{\partial}(x)\leq r_1\}$. Moreover, $\Sect\leq k$ on $\partial_{r_1}D$ for some positive constant $k$.
\medskip

Under Condition (A), based on the Hessian comparison theorem, F.-Y. Wang then constructs a function $h$ satisfying the necessary properties of Lemma \ref{esti:local-time}
(see \cite[p.~1436]{W07} or \cite[Theorem 3.2.9]{Wbook2} for the notation and the precise result),
along with explicit upper bounds for $\|h\|_{\infty}$ and the constant $K_{h, \alpha}$.
Modifying his construction one may take
\begin{align}\label{def-h}
\log h(x)=\frac{1}{\Lambda_0}\int_0^{\rho_{\partial}(x)}\l(\ell(s)-\ell(r_0)\r)^{1-n}\, d s \int_{s\wedge r_0}^{r_0}\l(\ell(u)-\ell(r_0)\r)^{n-1}\,du
\end{align}
where $\ell=\ell_{\sigma, \ell}$ is defined in \eqref{fun-h}, $r_0:=r_1\wedge \ell^{-1}(0)$ and $$ \Lambda_0:=(1-\ell(r_0))^{1-n}\int_0^{r_0}\l(\ell(s)-\ell(r_0)\r)^{n-1}\, d s.$$
Then from the proof of \cite[Theorem 1.1]{Wa05}, we get:
\begin{align}\label{esti-Kh}
K_{h,\alpha}\leq {K}_{\alpha}:=\frac{n}{r_0}+\alpha \quad \mbox{and} \quad \|h\|_{\infty}\leq 
\e^{\frac12 n r_0}.
\end{align}

Using the $h$ constructed above, we are now able to complete the proof of Theorem \ref{main-theorem1}.

\begin{proof}[Proof of Theorem \ref{main-theorem1}]
Using $h$  defined in \eqref{def-h} and substituting the estimates \eqref{esti-Kh}, we
 replace 
 $$K_{h,2\sigma},\ \, \|h\|_{\infty},$$
 by
 $$\frac{n}{r_0}+2\sigma,\ \, \e^{nr_0/2},$$
 respectively.
 By Lemma \ref{Hessian Comparison}, we see that the upper bound $\alpha$ in Theorem \ref{condition-rho1}
 can be chosen as $2(n-1)\max\{\sigma, \sqrt{k}\}$.
 This completes the proof of inequality \eqref{main-ineq1}.
\end{proof}

\section{Hessian estimates on Neumann eigenfunctions of Laplacian}\label{Neuman-Section}

We also use a stochastic approach to prove Theorem \ref{Neumann-them}. Let us first recall the
Hessian formulas for the Neumann semigroups, established recently in \cite{CT22}. 
The reflecting Brownian motion on $D$ with generator $\frac12\Delta$ satisfies the SDE
$$d X_t=\ptr_t\circ\,d B_t^x+\frac12N(X_t)\, d l_t,\quad X_0=x,$$
where $B_t^x$ is a standard Brownian motion on the Euclidean space $T_xD\cong \R^n$.
We write again $X_t=X_t(x)$ to indicate the starting point $x\in D$ (which may be on the boundary $\partial D$).
Here $\ptr_t: T_xD\rightarrow T_{X_t(x)}D$ denotes the $\nabla$-parallel
transport along $X_t(x)$ and
$l_t$ the local time of $X_t(x)$ supported on $\partial D$.
Note that the reflecting Brownian motion $X_t(x)$ is defined for all $t\geq0$. 

Suppose that $\tilde{Q}_t\colon T_xD\rightarrow T_{X_t(x)}D$ satisfies
\begin{align}\label{eq-tildeQ}
 \text{D}\tilde{Q}_t=-\frac12\Ric^{\sharp}(\tilde{Q}_t)\, d t+\frac12(\nabla N)^ {\sharp}(\tilde{Q}_t)\, d l_t, \quad \tilde{Q}_0=\id.
\end{align}
For $k\in C^1_b([0,\infty);\R)$ define an operator-valued process
$\tilde{W}_t^k\colon T_xD\otimes T_xD\rightarrow T_{X_t(x)}D$ as solution to the following  covariant It\^{o} equation
\begin{align}\label{eq-tildeW}
\text{D}\tilde{W}_t^k(v, w)&=R(\ptr_t\, d B_t, \tilde{Q}_t(k(t)v))\tilde{Q}_t(w)\notag\\
&\quad-\frac12(\bd^*R+\nabla \Ric)^{\sharp}(\tilde{Q}_t(k(t)v), \tilde{Q}_t(w))\, d t\notag\\
&\quad-\frac12(\nabla^2N-R(N))^{\sharp}(\tilde{Q}_t(k(t)v), \tilde{Q}_t(w))\,d l_t\notag\\
&\quad-\frac12\Ric^{\sharp}(\tilde{W}_t^k(v, w))\,d t+\frac12(\nabla N)^{\sharp}(\tilde{W}_t^k(v, w))\,d l_t,
\end{align}
with initial condition $\tilde{W}_0^k(v,w)=0$.

\begin{theorem}[\cite{CT22}]\label{th5}
Let $D$ be a compact Riemannian manifold with
boundary $\partial D$. Let $X(x)$ be the reflecting Brownian motion on $D$ with
starting point $x$ (possibly on the boundary) and denote by $P_tf(x)=\E[f(X_t(x)]$
the corresponding Neumann semigroup acting on $f\in \mathcal{B}_b(D)$.
Then, for $v\in T_xD$, $t\geq 0$ and $k\in C^1_b([0,\infty); \R)$,
 \begin{equation*}
    \Hess {P_tf}(v,v)=\E\l[-d f(\tilde{Q}_t(v))\int_0^t\langle \tilde{Q}_s(\dot{k}(s)v),
    \ptr_s d  B_s\rangle+ d f( \tilde{W}_t^k(v,v))\r].
  \end{equation*}
\end{theorem}

Estimating  $\tilde{W}^k$ and $ \tilde{Q}$ explicitly, we can get pointwise bounds
for the Hessian of Neumann eigenfunctions.

\begin{corollary}\label{th-esti-HS}
We keep the assumptions of Theorem \ref{th5}.  Let $K_0, K_1, K_2$ and $\sigma_1, \sigma_2$ be non-negative constants such that 
$\Ric\geq -K_0$, $|R|\leq K_1$ and $|\bd^* R+\nabla \Ric|\leq K_2$ on $D$, and $\II \geq -\sigma_1$,
$|\nabla^2 N+R(N)|<\sigma_2$ on the boundary $\partial D$. 
 Then, 
for $(\phi,\lambda)\in\Eig_N(D)$,
\begin{align*}
|\Hess \phi|(x)\leq & \e^{(\frac12\lambda +K_0)t}\E[\e^{\sigma_1 l_t}]\l(\frac{1}{\sqrt{t}}+K_1 \sqrt{t}+\frac{K_2}{2} t \r)\|\bd \phi\|_{\infty}\\
& +\frac{\sigma_2}{2} \e^{(K_0+\frac{\lambda}{2})t}\E \l(\e^{\frac{1}{2}\sigma_1 l_t}\int_0^t\e^{\frac{1}{2}\sigma_1 l_s}\, d l_s\r)\|\bd \phi\|_{\infty}.
\end{align*}
\end{corollary}
\begin{proof}
By \cite[Theorem 4.1]{CT22} the Hessian of the semigroup can be estimated as
\begin{align*}
|\Hess {P_tf}|& \leq \l(K_1 \sqrt{t}+\frac{K_2}{2} t+\frac{1}{\sqrt{t}} \r) \E\l[\e^{\sigma_1 l_t}\r]\e^{K_0t}  \|\nabla f\|_{\infty} \\
&\quad  +\frac{\sigma_2}{2}  \E \l(\e^{\frac{1}{2}\sigma_1 l_t}\int_0^t\e^{\frac{1}{2}\sigma_1 l_s}\, d l_s\r)\e^{K_0t}\|\nabla f\|_{\infty}.
\end{align*}
We complete the proof by observing that $P_t\phi=\e^{-\lambda t/2}\phi$.
\end{proof}

Combining Theorem \ref{th-esti-HS} and Lemma \ref{esti:local-time}, we are now in a position to prove
Theorem \ref{Neumann-them}. 
\begin{theorem}\label{Neumann-th0}
Let $D$ be an $n$-dimensional compact Riemannian manifold with
boundary $\partial D$. Let $K_0,K_1,K_2, \sigma_1, \sigma_2$ be non-negative constants such that $\Ric\geq -K_0$, $|R|\leq K_1$ and
$|\bd^* R+\nabla \Ric|\leq K_2$ on $D$, and that
$\II\geq -\sigma_1$ and $|\nabla^2 N-R(N)|\leq \sigma_2$ on the
boundary $\partial D$.  For $h\in C^{\infty}(D)$ with
$\min_Dh =1$ and $N\log h|_{\partial D}\geq 1$, let
$K_{h,\alpha}:=\sup_{D}\{-\Delta \log h+\alpha |\nabla \log h|^2\}$
with $\alpha$ a non-negative constant. Then for any non-trivial
$(\phi,\lambda) \in \Eig_N(\Delta)$, 
\begin{align*}
\frac{\|\Hess \,\phi\|_{\infty}}{\|\phi\|_{\infty}}\leq C_{N,\lambda}(D)\lambda
\end{align*}
where
\begin{align*}
C_{N,\lambda}(D)=&\e\l(1+\frac{K_1+2K_0+2\sigma_1 K_{h,2\sigma_1}}{\lambda}+\frac{K_2+2\sigma_2 K_{h,2\sigma_1}}{\lambda\sqrt{2\lambda+4K_0+4\sigma_1 K_{h,2 \sigma_1}}}\r)\|h\|_{\infty}^{3\sigma_1}\\
&+\frac{ \sigma_2 \e}{\lambda}\sqrt{2\lambda+4K_0+4\sigma_1 K_{h,2\sigma_1}}\|h\|_{\infty}^{3\sigma_1}\ln \|h\|_{\infty};
\end{align*}
\end{theorem}

\begin{proof}
By Lemma \ref{esti:local-time}, we have
\begin{align*}
 \E[\e^{\sigma_1 l_t}]
 \leq \E[\e^{\sigma_1 l_t}] \leq \|h\|_{\infty}^{2\sigma_1 }\exp \l(\sigma_1K_{h,2\sigma_1} t\r),
\end{align*}
and 
\begin{align*}
\E[\e^{\sigma_1 l_t}]\leq \|h\|_{\infty}^{2\sigma_1 }\exp \l(\sigma_1K_{h,2\sigma_1 } t\r).
\end{align*}
Moreover, we observe that
\begin{align*}
\E\l[\e^{\frac{1}{2}\sigma_1l_t}\int_0^t \e^{\frac{1}{2}\sigma_1l_s}\, d l_s\r]&\leq \frac{2(\E[\e^{(\sigma_1+\eps)l_t}]-1)}{\sigma_1+\eps}\\
&\leq \frac{2}{\sigma_1+\eps} \l(\|h\|_{\infty}^{2(\sigma_1+\eps)}\exp\l((\sigma_1+\eps)K_{h, 2(\sigma_1+\eps)}t\r)-1\r)\\
&\leq \frac{2}{\sigma_1+\eps}\l(\|h\|_{\infty}^{2(\sigma_1+\eps)}\exp\l((\sigma_1+\eps)K_{h, (\sigma_1+\eps)}t\r)-1\r)\\
&\leq \frac{2}{\sigma_1+\eps}\l(\|h\|_{\infty}^{2(\sigma_1+\eps)}-1\r)+\frac{2}{\sigma_1+\eps}\|h\|_{\infty}^{2(\sigma_1+\eps)}\l[\exp\l((\sigma_1+\eps)K_{h, 2(\sigma_1+\eps)}t\r)-1\r]\\
&\leq 4\|h\|_{\infty}^{2(\sigma_1+\eps)} \ln \|h\|_{\infty}+2\|h\|_{\infty}^{2(\sigma_1+\eps)}\exp\l((\sigma_1+\eps)K_{h, 2(\sigma_1+\eps)}t\r)K_{h, 2(\sigma_1+\eps)}t.
\end{align*}
Letting $\eps$ tend to $0$, we arrive at
\begin{align*}
\E\l[\e^{\frac{1}{2}\sigma_1l_t}\int_0^t \e^{\frac{1}{2}\sigma_1l_s}\, d l_s\r]&\leq 4\|h\|_{\infty}^{2\sigma_1} \ln \|h\|_{\infty}+2\|h\|_{\infty}^{2\sigma_1}\exp\l(\sigma_1K_{h, 2\sigma_1}t\r)K_{h,2\sigma_1}t.
\end{align*}
Therefore, combining this with Theorem \ref{th-esti-HS}, we obtain
\begin{align*}
\frac{\|\Hess \phi\|_{\infty}}{\|\bd \phi\|_{\infty}}&\leq  \e^{(\frac12\lambda +K_0)t}\l(\frac{1}{\sqrt{t}}+K_1 \sqrt{t}+\frac{K_2}{2} t \r)\|h\|_{\infty}^{2\sigma_1 }\exp \l(\sigma_1K_{h,2\sigma_1} t\r)\\
&\quad +\sigma_2\e^{(\frac12\lambda+K_0)t}\l[2 \ln \|h\|_{\infty}+K_{h,\sigma_1}t\r]\|h\|_{\infty}^{2\sigma_1 }\exp \l(\sigma_1K_{h,2\sigma_1} t\r)\\
&\leq  \e^{(\frac12\lambda +K_0)t}\l(\frac{1}{\sqrt{t}}+K_1 \sqrt{t}+\frac{K_2}{2} t \r)\|h\|_{\infty}^{2\sigma_1 }\exp \l(\sigma_1K_{h,2\sigma_1} t\r)\\
&\quad +\sigma_2\e^{(\frac12\lambda+K_0)t}\l[2 \ln \|h\|_{\infty}+K_{h,\sigma_1}t\r]\|h\|_{\infty}^{2\sigma_1 }\exp \l(\sigma_1K_{h,2\sigma_1} t\r).
\end{align*}
Let $t=\l(\lambda+2K_0+2\sigma_1K_{h,2\sigma_1 }\r)^{-1}$. Then we get
\begin{align*}
\frac{\|\Hess \phi\|_{\infty}}{\|\bd \phi\|_{\infty}} &\leq  \Bigg(\frac{K_1}{\sqrt{\lambda+2K_0+2\sigma_1K_{h,2\sigma_1}}}+\sqrt{\lambda+2K_0+2\sigma_1K_{h,2\sigma_1}}\\
&\qquad +\frac{K_2+2\sigma_2K_{h,\sigma_1}}{2(\lambda+2K_0+2\sigma_1K_{h,2\sigma_1 })}+2\sigma_2\ln\|h\|_{\infty} \Bigg)\|h\|_{\infty}^{2\sigma_1} \sqrt{ \e}.
\end{align*}
On the other hand,  from \cite{ATW}, it has already been shown that
\begin{align*}
\frac{\|\bd \phi \|_{\infty}}{\|\phi\|_{\infty}}\leq \frac{1}{\sqrt{t}}\E[\e^{\sigma_1 l_t}]^{1/2}\e^{\frac{1}{2}(K_0+\lambda)t}\leq  \frac{1}{\sqrt{t}}\|h\|_{\infty}^{\sigma_1 }\exp \l(\frac{1}{2}(\lambda+\sigma_1K_{h,2\sigma_1 } +K_0)t\r).
\end{align*}
Let $t=\l(\lambda+K_0+\sigma_1K_{h,2\sigma_1 }\r)^{-1}$. Then we get
\begin{align*}
\frac{\|\bd \phi \|_{\infty}}{\|\phi\|_{\infty}}\leq \sqrt{\lambda+K_0+\sigma_1K_{h,2\sigma_1 }} \|h\|_{\infty}^{\sigma_1} \sqrt{\e}.
\end{align*}
We then conclude that 
\begin{align*}
\frac{\|\Hess \phi\|_{\infty}}{\|\phi\|_{\infty}}& \leq  \Bigg(\lambda+K_1+2K_0+2\sigma_1K_{h,2\sigma_1} +\frac{K_2+2\sigma_2K_{h,\sigma_1}}{2\sqrt{\lambda+2K_0+2\sigma_1K_{h,2\sigma_1 }}}\\
&\qquad+2\sigma_2\ln\|h\|_{\infty}\sqrt{\lambda+K_0+\sigma_1K_{h,2\sigma_1 }} \Bigg)\|h\|_{\infty}^{3\sigma_1} \e.
\end{align*}

\end{proof}

\begin{proof}[Proof of Theorem \ref{Neumann-them}]
From the conditions we see that  Condition ${\bf (A)}$ is satisfied.
Then, the Hessian estimate of Neumann eigenfunctions in Theorem \ref{Neumann-th0}  remain valid by substituting the $h$ defined in \eqref{def-h}. Then  under replacing
$$ K_{h,\alpha}\ \mbox{ and }\ \|h\|_{\infty}\ $$
by
$$ {K}_{\alpha}:=\frac{n}{r_0}+\alpha\ \mbox{ and } \ \e^{n r_0/2}$$
respectively, the conclusion is just listed in Theorem \ref{Neumann-them}.
\end{proof}

\subsection*{Conflict of Interest and Ethics Statements}
On behalf of all authors, the corresponding author Li-Juan Chen declares that there is no conflict of interest. Data sharing is not applicable to this article as no datasets were created or analysed in this study.

\bibliographystyle{amsplain}%

\bibliography{Hessian-eigenfunctions}

\end{document}